\PassOptionsToPackage{english}{babel}
\documentclass[11pt]{article}
\usepackage{amsmath, amsfonts, amssymb, amsthm,  graphicx, mathtools, enumerate, dsfont}

\usepackage[blocks, affil-it]{authblk}

\allowdisplaybreaks

\usepackage[square]{natbib}
\usepackage[CJKbookmarks=true,
            bookmarksnumbered=true,
			bookmarksopen=true,
			colorlinks=true,
			citecolor=red,
			linkcolor=blue,
			anchorcolor=red,
			urlcolor=blue]{hyperref}
\usepackage[usenames]{color}

\usepackage[letterpaper, left=1.2truein, right=1.2truein, top = 1.2truein, bottom = 1.2truein]{geometry}
\usepackage[ruled, vlined, lined, commentsnumbered]{algorithm2e}

\usepackage{prettyref,soul}
\usepackage[ngerman]{babel}
\usepackage{bbm}
\usepackage{hyperref}
    \hypersetup{ colorlinks = true, linkcolor = blue, citecolor = blue}

\newtheorem{lemma}{Lemma}
\newtheorem{proposition}{Proposition}
\newtheorem{thm}{Theorem}

\newtheorem{corollary}{Corollary}

\makeatletter
\newcommand{\neutralize}[1]{\expandafter\let\csname c@#1\endcsname\count@}
\makeatother

\newrefformat{eq}{(\ref{#1})}
\newrefformat{chap}{Chapter~\ref{#1}}
\newrefformat{sec}{Section~\ref{#1}}
\newrefformat{algo}{Algorithm~\ref{#1}}
\newrefformat{fig}{Fig.~\ref{#1}}
\newrefformat{tab}{Table~\ref{#1}}
\newrefformat{rmk}{Remark~\ref{#1}}
\newrefformat{clm}{Claim~\ref{#1}}
\newrefformat{def}{Definition~\ref{#1}}
\newrefformat{cor}{Corollary~\ref{#1}}
\newrefformat{lmm}{Lemma~\ref{#1}}
\newrefformat{lemma}{Lemma~\ref{#1}}
\newrefformat{prop}{Proposition~\ref{#1}}
\newrefformat{app}{Appendix~\ref{#1}}
\newrefformat{ex}{Example~\ref{#1}}
\newrefformat{exer}{Exercise~\ref{#1}}
\newrefformat{soln}{Solution~\ref{#1}}
\newrefformat{cond}{Condition~\ref{#1}}

\def\P{{\mathbb P}}

\def\text#1{\mbox{\rm #1}}

\newcommand{\RR}{\mathbb{R}}

\newcommand{\supp}{{\rm supp}}

\newcommand{\Z}{\mathbb{Z}}

\newcommand{\E}{\mathbb{E}}
\renewcommand{\P}{\mathbb{P}}

\newcommand{\cov}{\text{Cov}}
\newcommand{\var}{\text{Var}}

\newcommand{\fm}{f_{j_{1:m}}}

\newcommand{\diam}{\text{diam}}

\renewcommand{\exp}{\textup{exp}}
\renewcommand{\cal}{\mathcal}
\newcommand{\define}{:=}

\newcommand{\ms}{\quad~}
\newcommand{\td}{\widetilde}

\DeclarePairedDelimiter{\abs}{\lvert}{\rvert}
\DeclarePairedDelimiter{\bbrace}{\lbrace}{\rbrace}
\DeclarePairedDelimiter{\parr}{(}{)}
\DeclarePairedDelimiter{\fence}{[}{]}

\DeclarePairedDelimiter{\nm}{\|}{\|}

\let\widehat\widehat

\usepackage{xr}
\begin{document}
\selectlanguage{english}

\setlength{\abovedisplayskip}{5pt}
\setlength{\belowdisplayskip}{5pt}
\setlength{\abovedisplayshortskip}{5pt}
\setlength{\belowdisplayshortskip}{5pt}

\title{Exponential inequalities for dependent V-statistics via random Fourier features}

\author[1]{Yandi Shen}
\author[2]{Fang Han}
\author[3]{Daniela Witten}
\affil[1]{
University of Washington
 \\
ydshen@uw.edu
}
\affil[2]{
University of Washington
 \\
fanghan@uw.edu
}
\affil[3]{
University of Washington
 \\
dwitten@uw.edu
}

\date{}

\maketitle

\begin{abstract}
We establish exponential inequalities for a class of V-statistics under strong mixing conditions. Our theory is developed via a novel kernel expansion based on random Fourier features and the use of a probabilistic method. This type of expansion is new and useful for handling many notorious classes of kernels. 
\end{abstract}

{\bf Keywords:}   dependent V-statistics, strong mixing condition, kernel expansion, random Fourier features

\section{Introduction}
\label{sec:intro}

Consider the following V-statistic of order $m$ generated by the symmetric kernel $f$,
\begin{align}
\label{eq:V}
V_n(f) \define \sum_{i_1, \ldots, i_m =1}^n f\parr*{X_{i_1},\ldots, X_{i_m}},
\end{align}
where $\{X_i\}_{i=1}^n$ is a stationary sequence with marginal measure $P$ on the $d$-dimensional real space $\RR^d$. The purpose of this paper is to establish exponential-type tail bounds for \eqref{eq:V} when $\{X_i\}_{i=1}^n$ are weakly dependent. 

In \eqref{eq:V}, if the summation is taken over $m$-tuples $(i_1,\ldots, i_m)$ of distinct indices, 
the resulting is a U-statistic. In many applications, the techniques of analyzing U- and V-statistics are the same. Non-asymptotic tail bounds and limiting theorems of V- and U-statistics in the i.i.d. case have also been extensively studied \citep{hoeffding1963probability,arcones1993limit,gine2000exponential,adamczak2006moment}. 

The analysis of V- and U-statistics when the observed data are no longer independent has attracted increasing attention in statistics and probability, with most of the efforts put on deriving limit theorems and bootstrap consistency. See, for instance, \cite{yoshihara1976limiting}, 
\cite{denker1982statistical}, \cite{denker1983u}, 
\cite{dehling1989empirical},  
\cite{dewan2001asymptotic},  \cite{hsing2004weighted},  \cite{dehling2006limit}, \cite{dehling2010central}, \cite{beutner2012deriving}, \cite{leucht2012degenerate}, \cite{leucht2013degenerate},  \cite{zhou2014inference}, \cite{atchade2014martingale}, among many others. 
However, there are few results on non-asymptotic concentration bounds for V- and U-statistics. 
Exceptions include \cite{borisov2015note} and \cite{han2018exponential}, who proved Hoeffding-type inequalities for U- and V-statistics under $\phi$-mixing conditions. There, the results either rely on assumptions difficult to verify, or are limited to nondegenerate ones. 

In this paper, we show that for a strongly mixing stationary sequence, exponential inequalities 
hold for a large class of V- and U-statistics. The main theorem is presented in Section \ref{sec:main}. We then illustrate the usefulness of our theory with examples and some further extensions in Section \ref{sec:example}. Detailed proof of the main theorem is given in Section \ref{sec:proof1}, with the rest of proofs given in Section \ref{sec:proof2}.

Notation used in the rest of the paper is as follows. $L_1(\RR^d)$ denotes the class of integrable functions in $\RR^d$, and for each $p\geq 1$, $\nm*{f}_{L_p} \define \Big\{\int_{\RR^d} \abs*{f(x)}^pdx\Big\}^{1/p}$. For a real vector $u\in\RR^d$, $\|u\|$ denotes its Euclidean norm. For two real numbers $a,b$, $a\vee b\define \max\{a,b\}$. 

\section{Main results}
\label{sec:main}

For two $\sigma$-algebras $\cal{A}$ and $\cal{B}$, the strong mixing coefficient is defined as
\begin{align*}
\alpha(\cal{A},\cal{B}) \define \sup_{A\in\cal{A},B\in\cal{B}}\abs*{\P(A\cap B) - \P(A)\P(B)}.
\end{align*}
A stationary sequence $\{X_i\}_{i\in\Z}$ 
is called \emph{strong mixing} (hereafter also called \emph{$\alpha$-mixing}) if 
\begin{align*}
\alpha(i) \define \alpha(\cal{M}_0,\cal{G}_i) \rightarrow 0 \text{ as } i\rightarrow\infty,
\end{align*}
where $\cal{M}_0 \define \sigma(X_j,j\leq 0)$ and $\cal{G}_i \define \sigma(X_j,j\geq i)$ for $i\geq 1$ are the $\sigma$-algebras generated by $\{X_j,j\leq 0\}$ and $\{X_j,j\geq i\}$ respectively.

We now introduce concepts in V-statistics. Let $\bbrace{\td{X}_i}_{i=1}^n$ be an i.i.d. sequence with $\td{X}_1$ identically distributed as $X_1$. The mean value of a symmetric kernel $f$ is defined as
\begin{align*}
\theta \define \theta(f) \define \E f\parr*{\td{X}_1,\ldots,\td{X}_m}.
\end{align*}
The kernel $f$ is called \emph{centered} if $\theta(f) = 0$, and \emph{degenerate of level $r-1$} ($2 \leq r \leq m$) if 
\begin{align*}
\E f\parr*{x_1,\ldots,x_{r-1},\widetilde{X}_r,\ldots, \widetilde{X}_m} = \theta
\end{align*}
for any $(x_1^\top,\ldots, x_{r-1}^\top)^\top\in\supp(P^{r-1})$, the support of the product measure $P^{r-1}$. 
The kernel $f$ is called \emph{fully degenerate} if it is degenerate of level $m-1$. 

When $f$ is degenerate of level $r - 1$, its Hoeffding decomposition takes the form
\begin{align*}
f(x_1,\ldots,x_m)  - \theta = \sum_{1\leq i_1<\ldots <i_r\leq m}f_r\parr*{x_{i_1},\ldots,x_{i_r}} + \ldots + f_m(x_1,\ldots,x_m),
\end{align*}
where $\{f_p\}_{p=r}^m$ are recursively defined as
\begin{equation}
\begin{aligned}
\label{eq:HD}
&f_1(x) \define g_1(x),\\
&f_p\parr*{x_1,\ldots,x_p} \define g_p(x_1,\ldots,x_p) -\sum_{k=1}^p f_1(x_k)- \ldots - \sum_{1\leq k_1<\ldots <k_{p-1} \leq p}f_{p-1}\parr*{x_{k_1},\ldots,x_{k_{p-1}}},
\end{aligned}
\end{equation}
for $p=2,\cdots,m$, with $\{g_p\}_{p=1}^m$ defined as $g_m \define f - \theta$, and 
\begin{align*}
g_p(x_1,\ldots,x_p) \define \E f\parr*{x_1,\ldots,x_p,\widetilde{X}_{p+1},\ldots,\widetilde{X}_m} - \theta
\end{align*}
for $1\leq p\leq m-1$. For each $1\leq p\leq m$ and $1\leq k\leq n$, we denote the V-statistic generated by $f_p$ and data $\{X_i\}_{i=1}^k$ by
\begin{align*}
V_k(f_p) \define \sum_{i_1,\ldots,i_p = 1}^k f_p(X_{i_1},\ldots,X_{i_p}).
\end{align*}

For a real function $g\in L_1(\RR^d)$, its Fourier transform is defined as
\begin{align*}
\widehat{g}(u) \define \int_{\RR^d} g(x)e^{-2\pi iu^\top x}dx,
\end{align*}
where $dx \define dx_1\ldots dx_d$. 


\begin{thm}
\label{thm:main}
Suppose $\{X_i\}_{i=1}^n$ in \eqref{eq:V} is part of a stationary sequence $\{X_i\}_{i\in\Z}$ that is geometrically $\alpha$-mixing with coefficient
\begin{align}
\label{eq:mixing_coef}
\alpha(i) \leq \gamma_1\exp(-\gamma_2 i)~~~\text{for all }i\geq 1,
\end{align}
where $\gamma_1,\gamma_2$ are two positive absolute constants.  Suppose $f\in L_1(\RR^{md})$ is continuous, and its Fourier transform $\widehat{f}$ satisfies
\begin{align}
\label{eq:fourier_moment}
\int_{\RR^{md}}\abs*{\widehat{f}(u)}\|u\|^qdu < \infty
\end{align}
for some $q \geq 1$. Then, there exists a positive constant $C=C(m,\gamma_1,\gamma_2)$ such that for each $1\leq p \leq m$, and any $x > 0$, 
\begin{align}
\label{eq:tail}
\P\bbrace*{n^{-p}\max_{1\leq k\leq n}\abs*{V_k(f_p)}\geq x} \leq 6\exp\parr*{-\frac{Cnx^{2/p}}{A_{p,n}^{1/p} + x^{1/p}M_{p,n}^{1/p}}}
\end{align}
with 
\begin{align}
\label{eq:constant}
A_{p,n} = 2^{2m}\nm*{\widehat{f}}_{L_1}^2\bbrace*{\frac{64\gamma_1^{1/3}}{1-\exp(-\gamma_2/3)}+\frac{(\log n)^4}{n}}^p~~{\rm and}~~M_{p,n} = 2^m\nm*{\widehat{f}}_{L_1}(\log n)^{2p}.
\end{align}
\end{thm}

We remark that a maximal-type tail estimate for $V_n(f)$ in \eqref{eq:V} can be obtained in a straightforward manner by assembling the tail estimate in \eqref{eq:tail} for each $r\leq p\leq m$, where $r$ is the degenerate level of $f$. Indeed, for any $1\leq k\leq n$ and $1\leq p \leq m$, we have by the symmetry of $f$
\begin{align*}
\sum_{1\leq i_1,\ldots,i_m\leq k}\sum_{1\leq j_1<\ldots<j_p\leq m}f_p(X_{i_{j_1}},\ldots,X_{i_{j_p}}) &= {m\choose p}\sum_{1\leq i_1,\ldots,i_m\leq k}f_p(X_{i_1},\ldots,X_{i_p})\\
& = {m\choose p}k^{m-p}\sum_{1\leq i_1,\ldots,i_p\leq k}f_p(X_{i_1},\ldots,X_{i_p}).
\end{align*}
This entails that
\begin{align*}
&\ms n^{-m}\max_{1\leq k\leq n}\abs*{\sum_{1\leq i_1,\ldots,i_m\leq k}\Big(f(X_{i_1},\ldots,X_{i_m}) - \theta\Big)}\\
& = n^{-m}\max_{1\leq k\leq n}\abs*{\sum_{1\leq i_1,\ldots,i_m\leq k}\sum_{p=r}^m\sum_{1\leq j_1<\ldots<j_p\leq m}f_p(X_{i_{j_1}},\ldots,X_{i_{j_p}})}\\
&\leq n^{-m}\max_{1\leq k\leq n}\sum_{r\leq p\leq m}{m\choose p}k^{m-p}\abs*{\sum_{1\leq i_1,\ldots,i_p\leq k}f_p(X_{i_1},\ldots,X_{i_p})}\\
&\leq \sum_{r\leq p\leq m} {m\choose p}n^{-p}\max_{1\leq k\leq n} \abs*{V_k(f_p)}.
\end{align*}
Therefore, by adjusting the constant $C$ in \eqref{eq:tail}, we obtain that
\begin{align*}
\P\parr*{n^{-m}\max_{1\leq k\leq n}\abs*{\sum_{1\leq i_1,\ldots,i_m\leq k}\Big(f(X_{i_1},\ldots,X_{i_m}) - \theta\Big)} \geq x}\leq 6\sum_{p=r}^m \exp\parr*{-\frac{Cnx^{2/p}}{A_{p,n}^{1/p} + x^{1/p}M_{p,n}^{1/p}}}.
\end{align*}

%

We now provide a proof sketch of Theorem \ref{thm:main} with a focus on the technical novelties. One key step in our proof is to find a uniform approximation $\td{f} = \td{f}(;t,M)$ of the original kernel $f$ under any prescribed accuracy $t$ such that (i) $\abs*{\td{f} - f}\leq t$ uniformly over a large enough compact set $[-M,M]^{md}$, and (ii) $\td{f}$ admits the following tensor expansion
\begin{align}
\label{eq:expansion}
\td{f}(x_1,\ldots,x_m) = \sum_{j_1,\ldots,j_m= 1}^K f_{j_1,\ldots,j_m}e_{j_1}(x_1)\ldots e_{j_m}(x_m). 
\end{align} 
Here, $K$ is a positive integer that depends on both the approximation error $t$ and the range of approximation $M$, $\{f_{j_1,\ldots,j_m}\}_{j_1,\ldots,j_m=1}^K$ is a real sequence, and $\{e_j(\cdot)\}_{j=1}^K$ is a set of uniformly bounded real bases. Once such an $\td{f}$ is found, a truncation argument will yield the proximity between $\{f_p\}_{p=1}^m$ and $\{\td{f}_p(;t,M)\}_{p=1}^m$, the latter being the degenerate components of $\td{f}(;t,M)$ in its Hoeffding decomposition. Then, using each $\td{f}_p(;t,M)$ as a proxy, standard moment estimates with the aid of exponential inequalities for partial sum processes (cf. Corollary 24 in \cite{merlevede2013rosenthal}) will render a tail bound for each $\max_{1\leq k\leq n}\abs*{V_k(f_p)}$. 

The problem then boils down to finding such an $\td{f}$ with the tensor structure \eqref{eq:expansion}. One main difficulty in this step 
is to construct expansion bases $\{e_j(\cdot)\}_{j=1}^K$ that are uniformly bounded. 
Many classical approaches in multivariate function approximation are unable to provide a satisfactory answer to this problem. For example, uniform polynomial approximation by the Stone-Weierstrass theorem will have very poor performance, since high orders of the polynomials lead to a  large upper bound of the bases. The use of Lipschitz-continuous scale and wavelet functions, as exploited in \cite{leucht2012degenerate}, is also inappropriate for the same reason.  

Our solution is based on a probabilistic method, and especially, by realizing that the tensor decomposition \eqref{eq:expansion} is intrinsically connected to the idea of randomized feature mapping \citep{rahimi2008random} in the kernel learning literature. More specifically, when $f\in L_1(\RR^{md})$ is continuous and $\widehat{f}\in L_1(\RR^{md})$, the Fourier inversion formula implies that 
\begin{align*}
f(x_1,\ldots,x_m) = \int_{\RR^{md}} \widehat{f}(u_1,\ldots,u_m)e^{2\pi i(u_1^\top x_1+\ldots+u_m^\top x_m)}du_1\ldots du_m,
\end{align*}
where the right-hand side can be seen as the expectation of a Fourier basis with random frequency, which follows the sign measure of $\widehat{f}$. Due to the boundedness of the Fourier bases, Hoeffding's inequality guarantees an exponentially fast rate for a sample mean statistic of Fourier bases 
\begin{align*}
s_K(x_1,\ldots,x_m) \define \frac{1}{K}\sum_{j=1}^K\exp\Big\{2\pi i(u_{j,1}^\top x_1+\ldots+u_{j,m}^\top x_m)\Big\}
\end{align*}
to approximate $f$ at each fixed point $x\in\RR^{md}$. The elements $\exp\{2\pi i(u_{j,1}^\top x_1+\ldots+u_{j,m}^\top x_m)\}$ in $s_K(x_1,\ldots,x_m)$ naturally decompose to bounded basis functions of inputs $x_j$.  An entropy-type argument is then used so that we could prove the existence of a satisfactory set of bases such that the approximation holds uniformly over any compact set $[-M,M]^{md}$. The detailed proof will be given in Section \ref{sec:proof1}.

\section{Examples and extensions}
\label{sec:example}

Motivated by their wide applications in statistics and machine learning, we will put special focus on shift-invariant symmetric kernels in the case $m=2$ with $f(x,y) = f_0(x-y)$ for some $f_0:\RR^d\rightarrow\RR$.
%
We start with a corollary of Theorem \ref{thm:main} for such kernels. 
\begin{corollary}
\label{cor:smooth_1}
Let $\{X_i\}_{i=1}^n$ be as in Theorem \ref{thm:main}. Let $m = 2$ and the kernel $f$ be shift-invariant with $f(x,y) = f_0(x-y)$ for some $f_0:\RR^d\rightarrow \RR$. Suppose that $f_0\in L_1(\RR^d)$ is continuous, and its Fourier transform $\widehat{f}_0$ satisfies 
\begin{align}
\label{eq:fourier}
\int_{\RR^{d}}\abs*{\widehat{f}_0(u)}\|u\|^qdu < \infty
\end{align}
for some $q \geq 1$. Then, for $p=1,2$, the same tail bound in \eqref{eq:tail} holds with 
\begin{align*}
A_{p,n} = 16\nm*{\widehat{f}_0}_{L_1}^2\bbrace*{\frac{64\gamma_1^{1/3}}{1-\exp(-\gamma_2/3)}+\frac{(\log n)^4}{n}}^p~~{\rm and}~~M_{p,n} = 4\nm*{\widehat{f}_0}_{L_1}(\log n)^{2p}.
\end{align*}
\end{corollary}

In view of Bochner's theorem (cf. Section 1.4.3, \cite{rudin2017fourier}), Corollary \ref{cor:smooth_1} can be further simplified when the kernel is positive definite.
Recall that a real function $g_0:\RR^d\rightarrow\RR$ is said to be \emph{positive definite (PD)} if for any positive integer $n$ and real vectors $\{x_i\}_{i=1}^n\in\RR^d$, the matrix $A = (a_{i,j})_{i,j=1}^n$ with $a_{i,j} = g_0(x_i-x_j)$ is positive semi-definite (PSD). 

\begin{corollary}[]
\label{cor:smooth_2}
Let $\{X_i\}_{i=1}^n$ be as in Theorem \ref{thm:main}. Let $m = 2$ and the kernel $f$ be shift-invariant with $f(x,y) = f_0(x-y)$ for some $f_0:\RR^d\rightarrow \RR$. Suppose $f_0$ satisfies the conditions in Corollary \ref{cor:smooth_1} and is also PD. Then, for $p=1,2$, the same tail bound in \eqref{eq:tail} holds with
\begin{align*}
A_{p,n} = 4f_0(0)^2\bbrace*{\frac{64\gamma_1^{1/3}}{1-\exp(-\gamma_2/3)}+\frac{(\log n)^4}{n}}^p~~{\rm and}~~M_{p,n} = 2f_0(0)(\log n)^{2p}.
\end{align*}
Moreover, the same bound holds with the above $A_p$ and $M_p$ if $f_0$ only satisfies \eqref{eq:fourier} for some $0 < q < 1$, but is both PD and Lipschitz continuous. 
\end{corollary}

We now list several commonly-used kernels covered by Theorem \ref{thm:main} and the previous two corollaries.  
\begin{enumerate}
\item The $d$-dimensional Gaussian kernel $f(x,y) = f_0(x-y) = \exp\parr*{-\|x-y\|^2/2}$ is shift-invariant with $f_0$ being both Schwartz and PD, and $f_0$ satisfies \eqref{eq:fourier} for arbitrary $q\geq 1$. Thus, $f$ satisfies the conditions of  Corollary \ref{cor:smooth_2}.
\item For the $d$-dimensional Cauchy kernel $f(x,y) = f_0(x-y)$ with $f_0(x) = \prod_{\ell=1}^d 2/\parr*{1+x_\ell^2}$, $f_0$ is PD and its Fourier transform $\widehat{f}_0(u) = \exp\parr*{-\|u\|_1}$ satisfies \eqref{eq:fourier} for arbitrary $q \geq 1$. Therefore, it satisfies the conditions of Corollary \ref{cor:smooth_2}. 
\item The $d$-dimensional Laplacian kernel $f(x,y) = f_0(x-y) = \exp(-\|x-y\|_1)$ is shift-invariant and PD. The Fourier transform of $f_0$ is $\widehat{f}_0(u) = \prod_{\ell=1}^d \bbrace*{2/(1+u_\ell^2)}$, which has fractional moments and thus satisfies \eqref{eq:fourier} for any $0<q<1$. Since $f_0$ is both PD and Lipschitz, it satisfies the conditions in Corollary \ref{cor:smooth_2}.
\item The 1-dimensional ``hat" kernel: $f(x,y) = f_0(x-y)$ with $f_0(x)$ equal to $x+1$ for $-1\leq x\leq 0$, $1-x$ for $0\leq x\leq 1$ and $0$ otherwise. $f_0$ is PD and $1$-Lipschitz. Its Fourier transform is $\widehat{f}_0(u) = \bbrace*{1-\cos\parr*{2\pi u}}/(2\pi^2u^2)$ and thus has fractional moment. Therefore, $f_0$ satisfies \eqref{eq:fourier} for any $0 < q < 1$, and hence is also covered by Corollary \ref{cor:smooth_2}. 
\end{enumerate}
 
We then discuss extensions to Theorem \ref{thm:main}. The smoothness assumption \eqref{eq:fourier_moment} in Theorem \ref{thm:main} could be further relaxed by employing the standard smoothing technique through mollifiers. More precisely, we resort to an intermediate kernel $f_h$ between $f$ and $\td{f}$. It is constructed by convolving $f$ with the Gaussian mollifier with scale parameter $h$. The parameter $h$ controls the trade-off between approximation error and smoothness: small $h$ leads to finer approximation of $f$ by $f_h$, but makes $f_h$ less smooth and thus renders a larger constant in the tail bound. Theorem \ref{thm:main} is then applied on this intermediate kernel $f_h$ to obtain the tail bound.

As a particular example, the following corollary deals with Lipschitz kernels considered in \cite{leucht2012degenerate}. 
Introduce the following constant from integration with polar coordinates (with convention $(-1)!! = 0!! = 1$):
\begin{align}
\label{eq:polar_constant}
\Gamma(n) &\define 
\begin{cases}
\parr*{(n-2)!!}^{-1}(2\pi)^{\frac{n}{2}} & n \text{ is even}\\
\parr*{(n-2)!!}^{-1}2(2\pi)^{\frac{n-1}{2}} & n \text{ is odd}
\end{cases}.
\end{align}

\begin{corollary}
\label{cor:unif_cont}
Let $\{X_i\}_{i=1}^n$ be as in Theorem \ref{thm:main}. Suppose the kernel $f \in L_1(\RR^{md})$ is bounded, uniformly continuous, and its Fourier transform satisfies 
\begin{align}
\label{eq:fourier_tail}
\abs*{\widehat{f}(u)} \leq \frac{L}{1+\|u\|^{md+\varepsilon}}
\end{align}
for some $\varepsilon > 0$ and positive constant $L$. Then, for $1\leq p\leq m$, the bound in \eqref{eq:tail} holds with 
\begin{align*}
A_{p,n} = (1+\varepsilon^{-1})^22^{2m}c_1^2L^2\bbrace*{\frac{64\gamma_1^{1/3}}{1-\exp(-\gamma_2/3)}+\frac{(\log n)^4}{n}}^p~{\rm and}~M_{p,n} = (1+\varepsilon^{-1})2^mc_1L(\log n)^{2p},
\end{align*}
where $c_1 = \Gamma(md)$.
\end{corollary}

The tail condition in \eqref{eq:fourier_tail} is in general milder than \eqref{eq:fourier_moment} in Theorem \ref{thm:main}, and naturally arises in Fourier analysis (cf. Chapter 8.4 in \cite{folland2013real}). The following is the version of Corollary \ref{cor:unif_cont} for shift-invariant kernels.
\begin{corollary}
\label{cor:unif_cont_shift}
Suppose $\{X_i\}_{i=1}^n$ are as in Theorem \ref{thm:main}. Let $m=2$ and the kernel $f$ be shift-invariant with $f(x,y) = f_0(x-y)$ for some $f_0:\RR^d\rightarrow\RR$. Suppose that $f_0$ satisfies condition \eqref{eq:fourier_tail} (with $m = 1$) for some $\varepsilon > 0$ and positive constant $L$. Then, for $p=1,2$, the bound in \eqref{eq:tail} holds with 
\begin{align*}
A_{p,n} = 16(1+\varepsilon^{-1})^2c_1^2L^2\bbrace*{\frac{64\gamma_1^{1/3}}{1-\exp(-\gamma_2/3)}+\frac{(\log n)^4}{n}}^p~{\rm and}~M_{p,n} = 4(1+\varepsilon^{-1})c_1L(\log n)^{2p},
\end{align*}
where $c_1 = \Gamma(d)$.
\end{corollary}
Corollaries \ref{cor:unif_cont} and \ref{cor:unif_cont_shift} cover  the cosine kernel, defined as $f(x,y) \define f_0(x-y) \define \prod_{\ell=1}^d \cos(x_\ell-y_\ell)\mathbbm{1}\parr*{\abs*{x_\ell-y_\ell}\leq \pi/2}$. Consider the simple 1-dimensional case. Here, even though the trigonometric identity $\cos(x-y) = \cos(x)\cos(y) + \sin(x)\sin(y)$ gives a direct expansion of $\cos(x-y)$, there is no trivial expansion of the indicator $\mathbbm{1}\parr*{|x-y|\leq \pi/2}$. However, letting $f_0(x) = \cos(x)\mathbbm{1}\parr*{|x|\leq \pi/2}$, it is immediate that $f_0$ is 1-Lipschitz and thus uniformly continuous. Moreover, its Fourier transform is $\abs*{\widehat{f}_0(u)} = \abs*{2\cos(\pi^2 u)/(1-4\pi^2u^2)}$, and hence $f_0$ satisfies \eqref{eq:fourier_tail} with $\varepsilon = 1$ and $L = 2$. 






\section{Proof of Theorem \ref{thm:main}}
\label{sec:proof1}

We will use the following extra notation. For any real-valued function $f$ on $\RR^d$, $\nabla_xf$ is the gradient of $f$. For a set $A$, $|A|$ indicates its cardinality. For a subset $\cal{M}$ in $\RR^d$, we will use $\diam(\cal{M})$ to denote its diameter, i.e. $\diam(\cal{M})\define \sup_{x,y\in\cal{M}}\|x-y\|$. For a function $f$, we write $f(;\theta)$ to emphasize its dependence on some parameter $\theta$. For a measurable set $A$, we will use $\mathbbm{1}\{A\}$ to denote the indicator variable on the set $A$. For any positive integer $N$, we will use $[N]$ to denote the set $\{1,\ldots,N\}$.

As described in the proof sketch after Theorem \ref{thm:main}, we split the main part of the proof into the following lemmas. The first lemma finds a symmetric kernel $\td{f}$ with tensor decomposition \eqref{eq:expansion} that approximates $f$ uniformly over some prescribed set $[-M,M]^{md}$ and accuracy $t$. 

\begin{lemma}
\label{lemma:approximation}
Suppose the kernel $f\in L_1(\RR^{md})$ is continuous and satisfies condition \eqref{eq:fourier_moment} for some $q\geq 1$. Then, for any $M > 0$ and $t > 0$, there exists a symmetric function $\td{f}=\td{f}(;t,M)$ such that $\abs*{f(x_1,\ldots,x_d) - \td{f}(x_1,\ldots,x_d)}\leq t$ uniformly over all $(x_1^\top,\ldots,x_d^\top)^\top\in[-M,M]^{md}$, and $\td{f}$ satisfies \eqref{eq:expansion} for some positive integer $K = K(t,M)$, $\{f_{j_1,\ldots,j_m}\}_{j_1,\ldots,j_m=1}^K$, and $\{e_j(\cdot)\}_{j=1}^K$ such that
\begin{align}
\label{eq:expansion_constraint}
\sum_{j_1,\ldots,j_m=1}^K \abs*{f_{j_1,\ldots,j_m}}\leq F\quad \text{ and } \quad \sup_{1\leq j\leq K}\sup_{x\in\RR^d}\abs*{e_j(x)} \leq B
\end{align}
for some constants $F,B$ that do not depend on $M$ and $t$. In particular, one can take $F = 2^m\nm*{\widehat{f}}_{L_1}$ and $B = 1$. 
\end{lemma}

\begin{proof}
This proof adapts from that of Claim 1 in \cite{rahimi2008random}. 
Throughout the proof, $x_1,\ldots,x_m$ and $u_1,\ldots,u_m$ are real vectors in $\RR^d$, $dx = dx_1\ldots dx_d$, and $x,u$ will be real vectors in $\RR^{md}$. Let $\widehat{f}:\RR^{md}\rightarrow\mathbb{C}$ be the Fourier transform of $f$, that is,
\begin{align*}
\widehat{f}(u_1,\ldots,u_m) = \int_{\RR^{md}} f(x_1,\ldots,x_m)e^{-2\pi i(u_1^\top x_1+\ldots+u_m^\top x_m)}dx_1\ldots dx_m.
\end{align*}
Clearly, Condition \eqref{eq:fourier_moment} with some $q\geq 1$ implies that $\widehat{f}\in L_1\parr*{\RR^{md}}$. Since $f$ is continuous, by the Fourier inversion formula (see, for example, Chapter 6 of \cite{stein2011fourier}), we have
\begin{align*}
f(x_1,\ldots,x_m) = \int_{\RR^d} \widehat{f}(u_1,\ldots,u_m)e^{2\pi i(u_1^\top x_1+\ldots+u_m^\top x_m)}du_1\ldots du_m.
\end{align*}
Note that without the continuity of $f$, the above equation only holds almost surely with respect to the Lebesgue measure. Let $\widehat{f} = \widehat{g} + i\widehat{h}$ for real-valued functions $\widehat{g},\widehat{h}$, then since $f$ is real-valued, we have $f = I - II$, where
\begin{align*}
I &\define \int_{\RR^{md}} \widehat{g}(u_1,\ldots,u_m)\cos\bbrace*{2\pi\parr*{u_1^\top x_1+\ldots + u_m^\top x_m}}du_1\ldots du_m,\\
II &\define \int_{\RR^{md}} \widehat{h}(u_1,\ldots,u_m)\sin\bbrace*{2\pi\parr*{u_1^\top x_1+\ldots+u_m^\top x_m}}du_1\ldots du_m.
\end{align*}
We now approximate $I$ and $II$ separately. $I$ can be further written as $I = I_+ - I_-$, where
\begin{align*}
I_+ &\define \int_{[\widehat{g}>0]}\widehat{g}(u_1,\ldots,u_m)\cos\bbrace*{2\pi\parr*{u_1^\top x_1+\ldots + u_m^\top x_m}}du_1\ldots du_m,\\
I_- &\define \int_{[\widehat{g}<0]}-\widehat{g}(u_1,\ldots,u_m)\cos\bbrace*{2\pi\parr*{u_1^\top x_1+\ldots + u_m^\top x_m}}du_1\ldots du_m.
\end{align*}
Let $A^+_g \define \int_{[\widehat{g} > 0]}\widehat{g}(u)du$ and $A^-_g \define \int_{[\widehat{g} <0]}\parr*{-\widehat{g}(u)}du$, and note that $A^+_g$ and $A^-_g$ are both nonnegative and satisfy $A^+_g + A^-_g = \|\widehat{g}\|_{L_1} < \infty$ and $A^+_g - A^-_g = f(0)$, where we use the fact that $\widehat{g}\in L_1\parr*{\RR^{md}}$ since $\widehat{f}\in L_1\parr*{\RR^{md}}$. Then, we have
\begin{align*}
I &= A^+_g\cdot\E_u\fence*{\cos\bbrace*{2\pi\parr*{u_1^\top x_1 + \ldots u_m^\top x_m}}} - A^-_g\cdot\E_v\fence*{\cos\bbrace*{2\pi\parr*{v_1^\top x_1 + \ldots v_m^\top x_m}}}\\
&=: A^+_g\cdot k^+_g(x_1,\ldots,x_m) - A^-_g\cdot k^-_g(x_1,\ldots,x_m),
\end{align*}
where $(u_1^\top,\ldots,u_m^\top)^\top$ follows the distribution $\widehat{g}\mathbbm{1}\bbrace*{\widehat{g} > 0}/A^+_g$, and $(v_1^\top,\ldots,v_m^\top)^\top$ follows the distribution $-\widehat{g}\mathbbm{1}\bbrace*{\widehat{g}<0}/A^-_g$. Assume without loss of generality that $A^+_g > 0$ and $A^-_g > 0$. We now focus on $I_+$. For any compact subset $\cal{M}\subset \RR^{md}$, there exist $T$ Euclidean balls with radius $r$ that cover $\cal{M}$, where $T \leq \bbrace*{c\,\diam(\cal{M})/r}^{md}$ with $c = 3\sqrt{md/\pi}$. Denote $\{d_1,\ldots, d_T\}$ as the centers of these balls in $\RR^{md}$. Now choose an i.i.d. sample $\{(u_{i1}^\top,\ldots,u_{im}^\top)^\top\}_{i=1}^{D_1}$ from the distribution $\widehat{g}\mathbbm{1}\bbrace*{\widehat{g} > 0}/A^+_g$ with the sample size $D_1$ to be specified later. Then, for each center $d=(d_1^\top,\ldots,d_m^\top)^\top$ and any $t > 0$, it holds by Hoeffding's inequality that
\begin{align*}
\P\bbrace*{\abs*{\frac{1}{D_1}\sum_{i=1}^{D_1}\cos\bbrace*{2\pi\parr*{u_{i1}^\top d_1+\ldots +u_{im}^\top d_m}} - k^+_g(d_1,\ldots,d_m)}\geq \frac{t}{8}}\leq \exp\parr*{-\frac{D_1t^2}{128}}.
\end{align*}
Let $s_{D_1}(x_1,\ldots,x_m) \define \sum_{i=1}^{D_1}\cos\bbrace*{2\pi\parr*{u_{i1}^\top x_1+\ldots +u_{im}^\top x_m}}/D_1$ so that $k^+_g(x_1,\ldots,x_m) = \E_u\bbrace*{s_{D_1}(x_1,\ldots,x_m)}$. Then, for any $q\geq 1$, it holds that
\begin{equation}
\begin{aligned}
\label{eq:lip_constant}
\E\fence*{\sup_{x}\|\nabla_x \bbrace*{s_{D_1}(x)-k_g^+(x)}\|^q} &= \E\fence*{\sup_x\|\nabla_x s_{D_1}(x)-\E\nabla_xs_{D_1}(x)\|^q}\\
&\leq \E\fence*{\sup_x\bbrace*{\|\nabla_x s_{D_1}(x)\| + \E\parr*{\|\nabla_x s_{D_1}(x)\|}}^q}\\
&\leq 2^{q-1}\E\fence*{\sup_x\|\nabla_xs_{D_1}(x)\|^q + \sup_x\bbrace*{\E\parr*{\|\nabla_xs_{D_1}(x)\|}}^q}\\
&\leq 2^q \E\parr*{\sup_x\|\nabla_xs_{D_1}(x)\|^q},
\end{aligned}
\end{equation}
where in the first line we use the finiteness of $\int_{\RR^{md}}\abs*{\widehat{f}(u)}\|u\|du$ (guaranteed by Condition \eqref{eq:fourier_moment}) and dominated convergence theorem to exchange the derivative with expectation. Moreover, 
\begin{align*}
&\ms\E\parr*{\sup_x\|\nabla_xs_{D_1}(x)\|^q} = \E\fence*{\sup_x\nm*{\frac{1}{D_1}\sum_{i=1}^{D_1}2\pi u_i\cos\bbrace*{2\pi\parr*{u_i^\top x}}}^q}\\
&\leq (2\pi)^q \E\bbrace*{\parr*{\frac{1}{D_1}\sum_{i=1}^{D_1}\nm*{u_i}}^q}\leq (2\pi)^q \E\bbrace*{\frac{1}{D_1}\sum_{i=1}^{D_1}\nm*{u_i}^q}= (2\pi)^q\E\parr*{\nm*{u_1}^q},
\end{align*}
where we have used the finiteness of $\E\parr*{\|u_1\|^q}$ since $\int_{\RR^{md}}\abs*{\widehat{f}(u)}\|u\|^qdu < \infty$ and the convexity of the function $x^q$ when $q\geq 1$. Therefore, it holds that
\begin{align*}
\E\bbrace*{\sup_x\|\nabla_x \parr*{s_{D_1}(x) - k^+_g(x)}\|^q} \leq (4\pi)^q\E\parr*{\|u_1\|^q}, 
\end{align*}
and thus by Markov's inequality,
\begin{align*}
\P\parr*{\sup_x\|\nabla_x\bbrace*{s_{D_1}(x) - k^+_g(x)}\|\geq \frac{t}{8r}} \leq \parr*{\frac{32\pi r}{t}}^q\E\parr*{\|u_1\|^q}. 
\end{align*}
By the triangle inequality, the event $\bbrace*{\sup_{x\in \cal{M}}\abs*{s_{D_1}(x)-k^+_g(x)}\leq t/4}$ has greater probability than the following event
\begin{align*}
\Big\{\abs*{s_{D_1}(d) - k^+_g(d)}\leq t/8,\forall d\in\{d_1,\ldots,d_T\}\Big\} \bigcap \Big\{\sup_x\|\nabla_x\bbrace*{s_{D_1}(x) - k^+_g(x)}\|\leq t/(8r)\Big\}.
\end{align*}
Therefore, we have
\begin{align*}
\P\bbrace*{\sup_{x\in \cal{M}}\abs*{s_{D_1}(x) -k^+_g(x)}\geq \frac{t}{4}} \leq \parr*{\frac{c\,\diam(\cal{M})}{r}}^{md}\exp\parr*{-\frac{D_1t^2}{128}} + \parr*{\frac{32\pi r}{t}}^q\E\parr*{\|u_1\|^q}.
\end{align*}
Letting the right-hand side of the above inequality be of the form $\kappa_1r^{-md} + \kappa_2r^q$, and $r = (\kappa_1/\kappa_2)^{1/(q+md)}$, we have
\begin{align*}
\P\bbrace*{\sup_{x\in \cal{M}}\abs*{s_{D_1}(x) -k_g^+(x)}\geq \frac{t}{4}} \leq 2\bbrace*{\frac{32\pi\parr*{\E\|u_1\|^q}^{1/q}c\,\diam(\cal{M})}{t}}^{\frac{qmd}{q+md}}\exp\parr*{-\frac{D_1t^2}{128}\frac{q}{q+md}}.
\end{align*}
Now, using the fact
\begin{align*}
\E\parr*{\|u_1\|^q} = \int_{\RR^{md}} \|u\|^q\frac{\widehat{g}(u)\mathbbm{1}\bbrace*{\widehat{g}(u)>0}}{A^+_g}du \leq \frac{1}{A^+_g}\int_{\RR^{md}}\|u\|^q\abs*{\widehat{g}(u)}du \leq \frac{1}{A_g^+}\int_{\RR^{md}}\|u\|^q\abs*{\widehat{f}(u)}du,
\end{align*}
we conclude that there exists $\{u_i\}_{i=1}^{D_1}\in\RR^{md}$ such that uniformly over $\cal{M}$, it holds that
\begin{align*}
A^+_g\cdot\abs*{s_{D_1}(x) - k_g^+(x)} = \abs*{\frac{A^+_g}{D_1}\sum_{i=1}^{D_1}\cos\bbrace*{2\pi\parr*{u_i^\top x}} - A^+_g\cdot k^+_g(x)} \leq A^+_g\frac{t}{4}
\end{align*}
when $D_1$ is chosen such that
\begin{align*}
D_1 \geq C_1 \frac{md}{t^2}\log\bbrace*{\frac{\pi c\,\diam(\cal{M})\mu_q\parr*{\widehat{f}}}{(A^+_g)^{1/q}t}}
\end{align*}
for some sufficiently large constant $C_1 = C_1(t, f, \cal{M})$. Equivalently, it holds that 
\begin{align*}
\abs*{A^+_g\cdot s_{D_1}(x) - A^+_g\cdot k^+_g(x)}\leq t/4
\end{align*}
when $D_1$ is chosen such that
\begin{align*}
D_1 \geq C_1 \frac{md(A^+_g)^2}{t^2}\log\bbrace*{\frac{\pi c\,\diam(\cal{M})(A^+_g)^{1-1/q}\mu_q\parr*{\widehat{f}}}{t}}.
\end{align*}
Similarly, it can be shown that there exists $\{v_i\}_{i=1}^{D_2}\in\RR^{md}$ 
such that $\abs*{A^-_g\cdot s_{D_2}(x)-A^-_g\cdot k^-_g(x)}\leq t/4$ uniformly over $x\in\cal{M}$, where
\begin{align*}
s_{D_2}(x) = \frac{1}{D_2}\sum_{i=1}^{D_2}\cos\bbrace*{2\pi\parr*{v_i^\top x}},
\end{align*}
and $D_2$ is chosen such that
\begin{align*}
D_2 \geq C_2 \frac{md(A^-_g)^2}{t^2}\log\bbrace*{\frac{8\pi c\,\diam(\cal{M})(A^-_g)^{1-1/q}\mu_q\parr*{\widehat{f}}}{t}}
\end{align*}
for some sufficiently large constant $C_2 = C_2(t, f, \cal{M})$. Repeating this procedure for the approximation of $II$, then with $A^+_h,A^-_h,k^+_h,k^-_h$ similarly defined as $A^+_g,A^-_g,k^+_h,k^-_h$, we can find $s_{D_3}$ and $s_{D_4}$ which are sample means of sine functions such that $\abs*{A^+_h\cdot s_{D_3}(x) - A^+_h\cdot k^+_h(x)}\leq t/4$ and $\abs*{A^-_h\cdot s_{D_4}(x) - A^-_h\cdot k^-_h(x)}\leq t/4$ uniformly over all $x\in\cal{M}$, when the sample sizes $D_3$ and $D_4$ are respectively chosen such that
\begin{align*}
D_3 &\geq C_3 \frac{md(A^+_h)^2}{t^2}\log\bbrace*{\frac{\pi c\,\diam(\cal{M})(A^+_h)^{1-1/q}\mu_q\parr*{\widehat{f}}}{t}},\\
D_4 &\geq C_4 \frac{md(A^-_h)^2}{t^2}\log\bbrace*{\frac{\pi c\,\diam(\cal{M})(A^-_h)^{1-1/q}\mu_q\parr*{\widehat{f}}}{t}}
\end{align*}
for some sufficiently large constants $C_3,C_4$ that depend on $t,f,\cal{M}$. Putting together the pieces, we obtain that
\begin{align*}
\abs*{s_D(x) - f(x)} \define \abs*{\bbrace*{A^+_g\cdot s_{D_1}(x) - A^-_g\cdot s_{D_2}(x) - A^+_h\cdot s_{D_3}(x) + A^-_h\cdot s_{D_4}(x)} - f(x)}
\end{align*}
is smaller than $t$ when $D_1$-$D_4$ are chosen as above. Since 
\begin{align*}
A^+_g + A^-_g + A^+_h + A^-_h = \int_{\RR^{md}}\Big|\widehat{g}\Big| + \Big|\widehat{h}\Big| \leq \sqrt{2}\int_{\RR^{md}}\sqrt{\Big|\widehat{g}\Big|^2+\Big|\widehat{h}\Big|^2} = \sqrt{2}\nm*{\widehat{f}}_{L_1}
\end{align*}
and for each $u$, $\cos\bbrace*{2\pi\parr*{u^\top x}}$ can be written as at most $2^{m-1}$ linear combinations of the term $z_{u_1}\parr*{2\pi u_1^\top x_1}\ldots z_{u_m}\parr*{2\pi u_m^\top x_m}$, where $\{z_{u_i}(\cdot)\}_{i=1}^m$ is either the cosine or sine function.\\
 Therefore, taking $\cal{M} = [-M,M]^{md}$, it holds that $\abs*{s_D-f}\leq t$ uniformly over $[-M,M]^{md}$, and $s_D$ satisfies \eqref{eq:expansion_constraint} with constants $F = 2^m\nm*{\widehat{f}}_{L_1}$ and $B = 1$. Lastly, define the symmetrized version of $s_D$ to be
\begin{align*}
\widetilde{s}_D(x_1,\ldots,x_m) \define \frac{1}{m!}\sum_{\pi}s_D(\pi(x_1),\ldots,\pi(x_m)),
\end{align*}
where the summation is taken over all $m!$ permutations of $(x_1,\ldots,x_m)$. Then, due to the symmetry of $f$, $\abs*{\td{s}_D - f}\leq t$ uniformly over $[-M,M]^{md}$ and $\td{s}_D$ satisfies \eqref{eq:expansion_constraint} with the same $F,B$ as $s_D$. 
\end{proof}

The second lemma builds upon the previous one and guarantees the existence of an approximating kernel $\td{f}$ such that $f_p$ and $\td{f}_p$, the $p$th term in the Hoeffding decomposition of $f$ and $\td{f}$, are sufficiently close for each $1\leq p\leq m$.

\begin{lemma}
\label{lemma:hoeffding_approximation}
Suppose the kernel $f\in L_1(\RR^{md})$ is continuous and satisfies condition \eqref{eq:fourier_moment} for some $q\geq 1$. Then, for any $M > 0$ and $t > 0$, there exists a symmetric function $\td{f}=\td{f}(;t,M)$ such that $\td{f}$ satisfies all the properties in Lemma \ref{lemma:approximation}, and moreover, for each $1\leq p\leq m$, 
\begin{align*}
\abs*{f_p(x_1,\ldots,x_p) - \td{f}_p(x_1,\ldots,x_p)} \leq Ct
\end{align*}
uniformly over all $(x_1^\top,\ldots,x_p^\top)\in[-M,M]^{pd}$ for some positive constant $C = C(m)$.
\end{lemma}

\begin{proof}
To highlight dependence, for any $t_0 > 0$ and $M_0 > 0$, we will denote the approximating kernel in Lemma \ref{lemma:approximation} by $\td{f} = \td{f}(;t_0,M_0)$, so that $\abs*{f - \td{f}}\leq t_0$ uniformly over $[-M_0,M_0]^{md}$ and $\td{f}$ satisfies \eqref{eq:expansion_constraint} with $F = 2^m\nm*{\widehat{f}}_{L_1}$ and $B = 1$. This implies that
\begin{align}
\label{eq:majorant}
\sup_{(x_1^\top,\ldots,x_m^\top)^\top\in\RR^d} \abs*{\td{f}(x_1,\ldots,x_d;t_0,M_0)} \leq 2^m\nm*{\widehat{f}}_{L_1},
\end{align}
and, in particular, $\td{f}\in L_1(\RR^{md})$ under the product measure $P^m$. We will prove Lemma \ref{lemma:hoeffding_approximation} by choosing a $\td{f}(;t_0,M_0)$ with some $t_0$ and $M_0$ to be specified later that only depend on the prescribed $t$ and $M$. Again, to show clearly the dependence on $t_0$ and $M_0$, we write $\td{\theta}(t_0,M_0)$, $\td{f}_p(;t_0,M_0)$, and $\td{g}_p(;t_0,M_0)$ in the Hoeffding decomposition of $\td{f}(;t_0,M_0)$.
By definition, in order to show $\abs*{f_p - \widetilde{f}_p(;t_0,M_0)}\leq Ct$ over $[-M,M]^{pd}$ for some $C = C(m)$, it suffices to show that $\abs*{\theta - \widetilde{\theta}(t_0,M_0)}\leq 2t_0$ and $\abs*{g_p - \widetilde{g}_p(;t_0,M_0)}\leq 2t_0$ over $[-M_0,M_0]^{pd}$ as long as we choose $t_0 \leq t$ and $M_0 \geq M$.
For $\abs*{\theta - \widetilde{\theta}(t_0,M_0)}$, recalling that $\{\td{X}_i\}_{i=1}^m$ are i.i.d. with law $P$, one has
\begin{align*}
\abs*{\theta - \widetilde{\theta}(t_0,M_0)} &= \abs*{\E\bbrace*{f(\widetilde{X}_1,\ldots,\widetilde{X}_m) - \widetilde{f}(\widetilde{X}_1,\ldots,\widetilde{X}_m;t_0,M_0)}}\\
&\leq \E\bbrace*{\abs*{f - \widetilde{f}(;t_0,M_0)}\mathbbm{1}\bbrace*{(\widetilde{X}_1^\top,\ldots,\widetilde{X}_m^\top)^\top \in [-M_0,M_0]^{md}}} + \\
&\ms\E\bbrace*{\abs*{f - \widetilde{f}(;t_0,M_0)}\mathbbm{1}\bbrace*{(\widetilde{X}_1^\top,\ldots,\widetilde{X}_m^\top)^\top \notin [-M_0,M_0]^{md}}}\\
&\leq t_0 + \E\bbrace*{\abs*{f - \widetilde{f}(;t_0,M_0)}\mathbbm{1}\bbrace*{(\widetilde{X}_1^\top,\ldots,\widetilde{X}_m^\top)^\top \notin [-M_0,M_0]^{md}}},
\end{align*}
Now, by \eqref{eq:majorant}, the variable $\abs*{f - \widetilde{f}}\mathbbm{1}\bbrace*{(\widetilde{X}_1^\top,\ldots,\widetilde{X}_m^\top)^\top \notin [-M_0,M_0]^{md}}$ has an integrable majorant $|f| + 2^m\nm*{\widehat{f}}_{L_1}$ under the product measure $P^m$, and clearly converges to zero in probability as $M_0 \rightarrow \infty$. Thus, by choosing $t_0 = t$ and the dominated convergence theorem, there exists some $M_1 = M_1(t)$ such that for each $M_0 \geq M_1(t)$, $\abs*{\theta - \widetilde{\theta}(t,M_0)} \leq 2t$. With a similar argument, there exists some $M_2 = M_2(t)$ such that for any $M_0 \geq M_2(t)$ and $1\leq p\leq m$, it holds that
\begin{align*}
\abs*{g_p(x_1,\ldots,x_p) - \td{g}_p(x_1,\ldots,x_p;t,M_0)} \leq 2t, \quad \text{ for all } (x_1^\top,\ldots,x_p^\top)^\top \in[-M_0,M_0]^{pd}.
\end{align*} 
Therefore, by choosing $M_0 \define M \vee M_1(t) \vee M_2(t)$, one has $\td{f}(;t, M_0)$ satisfies all the desired properties. This completes the proof.
\end{proof}

The third lemma derives a maximal-type tail bound for each $\max_{1\leq k\leq n}\abs*{V_k(f_p)}$ when $f$ admits the tensor decomposition \eqref{eq:expansion}.

\begin{lemma}
\label{lemma:tail}
Suppose $\{X_i\}_{i=1}^n$ are as in Theorem \ref{thm:main}. Suppose the symmetric kernel $f:\RR^{md}\rightarrow\RR$ can be written as
\begin{align*}
f(x_1,\ldots,x_m) = \sum_{j_1,\ldots,j_m=1}^K f_{j_1,\ldots,j_m}e_{j_1}(x_1)\ldots e_{j_m}(x_m),
\end{align*}
where $K$ is some positive integer, $\{f_{j_1,\ldots,j_m}\}_{j_1,\ldots,j_m=1}^K$ is a real sequence, and $\{e_j(\cdot)\}_{j=1}^K$ is a set of real basis functions satisfying 
\begin{align*}
\sum_{j_1,\ldots,j_m=1}^K \abs*{f_{j_1,\ldots,j_m}}\leq F \quad \text{ and } \quad \sup_{1\leq j\leq K}\sup_{x\in\RR^d}\abs*{e_j(x)} \leq B
\end{align*}
for some positive constants $F$ and $B$. Let $\mu_a\define \sup_{1\leq j\leq K}\parr*{\E\abs*{e_j(X_1)}^a}^{1/a}$ for each $a\geq 1$. Then, there exists a positive constant $C = C(m,\gamma_1,\gamma_2)$ such that for any $1\leq p\leq m$, and any $x \geq 0$, 
\begin{align*}
\P\Big(n^{-p}\max_{1\leq k\leq n}\abs*{V_k(f_p)} \geq x\Big) \leq 6\exp\parr*{-\frac{Cnx^{2/p}}{A_{p,n}^{1/p} + x^{1/p}M_{p,n}^{1/p}}},
\end{align*}
where 
\begin{align*}
A_{p,n}= \mu_1^{2(m-p)}F^2\parr*{\sigma^2 + B^2(\log n)^4/n}^p,\quad M_{p,n} = \mu_1^{m−p}FB^p(\log n)^{2p},
\end{align*}
and $\sigma^2= 64\gamma_1^{1/3}\mu_3^2/(1-\exp(-\gamma_2/3))$.
\end{lemma}

\begin{proof}
Throughout the proof, let $C_i$'s be positive constants that only depend on $m,\gamma_1,\gamma_2$, and we will use the shorthand $f_{j_{a:b}}$ for $f_{j_a,\ldots,j_b}$ for positive integers $a < b$. We drop the dependence of $A_{p,n}$ and $M_{p,n}$ on $n$ for notational simplicity.

Fix a $1\leq p\leq m$ and we now derive the tail bound for $\max_{1\leq k\leq n}\abs*{V_k(f_p)}$. 
For the set of bases $\{e_j(\cdot)\}_{j=1}^K$ in the expansion of $f$, define $\widetilde{e}_j\define e_j - \E\bbrace*{e_j(X_1)}$ for $j\in[K]$. Since $f$ is symmetric, for any $(x_1^\top,\ldots,x_m^\top)^\top\in\RR^{md}$, $f(x_1,\ldots,x_m) = f(\pi(x_1),\ldots,\pi(x_m))$ for any permutation $\pi$ of $\{x_1,\ldots,x_m\}$. By the definition of $\{f_p\}_{p=1}^m$ in \eqref{eq:HD}, one can readily check that 
\begin{align*}
f_p(x_1,\ldots,x_p) = \sum_{j_1,\ldots,j_m=1}^K f_{j_1,\ldots,j_m} \E\parr*{e_{j_1}}\ldots\E\parr*{e_{j_{m-p}}}\widetilde{e}_{j_{m-p+1}}(x_1)\ldots\widetilde{e}_{j_{m}}(x_p),
\end{align*}
for $r\leq p\leq m$. Thus, we have
\begin{align*}
V_k(f_p) = \sum_{j_1,\ldots,j_m=1}^K \fm \E\parr*{e_{j_1}}\ldots\E\parr*{e_{j_{m-p}}}\bbrace*{\sum_{i=1}^k \widetilde{e}_{j_{m-p+1}}(X_i)}\ldots\bbrace*{\sum_{i=1}^k \widetilde{e}_{j_{m}}(X_i)}.
\end{align*}
Define, for each $j\in[K]$ and $k\in[n]$, 
\[
S_{k,j} \define \sum_{i=1}^k \widetilde{e}_{j}(X_i)~~~ {\rm and}~~~ Z_j \define \max_{1\leq k\leq n}\abs*{S_{k,j}}.
\]
Note that for each $j\in[K]$, $\{\widetilde{e}_{j}(X_i)\}_{i=1}^n$ is also geometrically $\alpha$-mixing. We now control each even order moment of $\max_{1\leq k\leq n}\abs*{V_k(f_p)}$. Let 
\[
T_p\define n^{-p}\max_{1\leq k\leq n}\abs*{V_k(f_p)}. 
\]
Define
\begin{align*}
\nu_{j} = C_3(n\sigma^2_{j}+B^2), \quad \nu = C_3(n\sigma^2+B^2), \quad c = C_4B(\log n)^2
\end{align*}
and $\sigma^2 = \sup_{j\in[K]}\sigma_j^2$, with
\begin{align*}
\sigma_j^2 \define \var\bbrace*{\widetilde{e}_j(X_1)} + 2\sum_{i>1}\abs*{\cov\bbrace*{\widetilde{e}_j\parr*{X_1},\widetilde{e}_j\parr*{X_i}}}.
\end{align*}
Integrating the tail estimate in Corollary 24 of \cite{merlevede2013rosenthal} and using Theorem 2.3 in \cite{boucheron2013concentration} yield that, for any positive integer $N$, by choosing  $C_4$ in $c$ to be sufficiently large,
\begin{align*}
\E(Z_j^{2pN}) &= (4pN)\int_0^\infty x^{2pN-1}\P\parr*{Z_j \geq x}dx\\
& = 4pN\cdot\parr*{\int_0^{CB\log n} x^{2pN-1}\P\parr*{Z_j \geq x}dx + \int_{CB\log n}^\infty x^{2pN-1}\P\parr*{Z_j \geq x}dx}\\
&\leq C(B\log n)^{2pN} + (pN)!(8\nu)^{pN} + (2pN)!(4c)^{2pN}\\
&\leq (pN)!(8\nu)^{pN} + (2pN)!(5c)^{2pN}.
\end{align*}
Then, employing a similar argument as in \cite{borisov2015note} (cf. Equation (12) therein), it holds that 
\begin{align*}
&\ms \E T_p^{2N} = \E\max_{1\leq k_1,\ldots,k_{2N}\leq n}\sum_{j_1,\ldots,j_{2mN}=1}^K f_{j_{1:m}}\ldots f_{j_{(2N-1)m+1:2mN}}\cdot \E\parr*{e_{j_1}}\ldots\E\parr*{e_{j_{m-p}}}\ldots\\
&\ms \E\parr*{e_{j_{(2N-1)m+1}}}\ldots\E\parr*{e_{j_{2Nm - p}}}\E\parr*{S_{k_1,j_{m-p+1}}\ldots S_{k_1,j_{m}}\ldots S_{k_{2N},j_{2Nm-p+1}}\ldots S_{k_{2N},j_{2Nm}}}\\
&\leq n^{-2Np}\mu_1^{2N(m-p)}\sum_{j_1,\ldots,j_{2mN}=1}^K \abs*{f_{j_{1:m}}}\ldots \abs*{f_{j_{(2N-1)m+1:2mN}}} \E\parr*{\abs*{Z_{j_{m+1-p}}}\ldots \abs*{Z_{j_{2mN}}}}\\
&\leq n^{-2Np}\mu_1^{2N(m-p)}\sum_{j_1,\ldots,j_{2mN}=1}^K \abs*{f_{j_{1:m}}}\ldots\abs*{f_{j_{(2N-1)m+1:2mN}}}\bbrace*{\E\parr*{Z_{j_{m+1-p}}^{2pN}}}^{\frac{1}{2pN}}\ldots\bbrace*{\E\parr*{Z_{j_{2mN}}^{2pN}}}^{\frac{1}{2pN}}\\
&\leq n^{-2Np}\mu_1^{2N(m-p)}F^{2N}\bbrace*{(pN)!(8\nu)^{pN} + (2pN)!(4c)^{2pN}},
\end{align*}
where in the second inequality we use the generalized H\"older inequality. By Stirling's approximation formula $\sqrt{2\pi}n^{n+1/2}e^{-n} \leq n! \leq en^{n+1/2}e^{-n}$, it holds that
\begin{align*}
\bbrace*{(pN)!}^{1/p} &\leq e^{1/p} (pN)^{N+1/2p} e^{-N} \leq C_5^N N^{N+1/2} e^{-N}\leq C_6^NN!.
\end{align*}
Similarly, we have $\bbrace*{(2pN)!}^{1/p}\leq C_6^{2N}(2N)!$. Thus we have
\begin{align*}
\E\parr*{T_p}^{\frac{2N}{p}} &\leq \parr*{\E T_p^{2N}}^{\frac{1}{p}}\\
& \leq n^{-2N}\mu_1^{\frac{2N(m-p)}{p}}F^{\frac{2N}{p}}\bbrace*{C_7^NN!\nu^N+C_8^{2N}(2N)!c^{2N}}.
\end{align*}
Now we control the Laplace transform of $T_p^{1/p}$, 
\begin{align}
\label{eq:bern_moment}
\ms&\E\parr*{e^{\lambda T_p^{1/p}}} = \sum_{N=0}^\infty \frac{\lambda^N}{N!}\E T_p^{N/p} \leq 3\sum_{N=0}^\infty \frac{\lambda^{2N}}{(2N)!}\E T_p^{2N/p}\notag\\
\leq& 3\bbrace*{\sum_{N=0}^\infty \frac{\lambda^{2N}}{(2N)!}C_7^Nn^{-2N}N!\mu_1^{\frac{2N(m-p)}{p}}F^{\frac{2N}{p}}\nu^N + \sum_{N=0}^\infty \lambda^{2N}n^{-2N}C_8^{2N}\mu_1^{\frac{2N(m-p)}{p}}F^{\frac{2N}{p}}c^{2N}},
\end{align}
where in the first inequality we use only the even moments with an absolute constant 3. For the first summand in \eqref{eq:bern_moment}, we have
\begin{align*}
\sum_{N=0}^\infty \lambda^{2N}\frac{N!}{(2N)!}C_7^Nn^{-2N}\mu_1^{2N(m-p)/p}F^{2N/p}\nu^N &\leq \sum_{N=0}^\infty \frac{(\lambda/n)^{2N}}{N!}2^{-N}C_{9}^N\mu_1^{2N(m-p)/p}F^{2N/p}\nu^N\\
&= \exp\bbrace*{C_{10}(\lambda/n)^2\mu_1^{2(m-p)/p}F^{2/p}\nu},
\end{align*}
where in the first line we use the relation $N!/(2N)!\leq 2^{-N}/N!$. For the second summand, we have
\begin{align*}
\sum_{N=0}^\infty \lambda^{2N}n^{-2N}C_8^{2N}\mu_1^{2N(m-p)/p}F^{2N/p}c^{2N} &= 1 + \frac{(\lambda/n)^2C_8^2\mu_1^{2(m-p)/p}F^{2/p}c^2}{1-(\lambda/n)^2C_8^2\mu_1^{2(m-p)/p}F^{2/p}c^2}\\
&\leq 1 + \frac{(\lambda/n)^2C_8^2\mu_1^{2(m-p)/p}F^{2/p}c^2}{1-(\lambda/n) C_8\mu_1^{(m-p)/p}F^{1/p}c}
\end{align*}
for $\lambda \leq n/\bbrace*{C_8\mu_1^{(m-p)/p}F^{1/p}c}$. Now using the relation $e^{x} + 1 + y \leq 2e^{x+y}$ which holds for all positive $x,y$, we have
\begin{align*}
\E\parr*{e^{\lambda T_p^{1/p}}} &\leq 6\exp\fence*{\frac{(\lambda/n)^2C_{11}\mu_1^{\frac{2(m-p)}{p}}F^{\frac{2}{p}}(\nu + c^2)}{2\bbrace*{1-(\lambda/n) C_8c\mu_1^{\frac{(m-p)}{p}}F^{\frac{1}{p}}}}} = 6\text{exp}\bbrace*{\frac{(\lambda/n)^2C_{11}nA_p^{1/p}}{2\parr*{1-(\lambda/n) C_8M_p^{1/p}}}}.
\end{align*}
Now, taking $\lambda = nx^{1/p}/(C_{11}nA_p^{1/p} + C_8M_p^{1/p}x^{1/p})$ in the exponential Markov inequality, we have
\begin{align*}
\P\parr*{T_p\geq x} \leq \exp(-\lambda x^{1/p})\E\parr*{e^{\lambda T_p^{1/p}}} \leq 6\exp\parr*{-\frac{C_{12}nx^{2/p}}{A_p^{1/p}+x^{1/p}M_p^{1/p}}}.
\end{align*}
Moreover, taking $\delta = 1$ in Theorem 3 of \cite{doukhan1994lecture}, we obtain
\begin{align*}
&\ms\var\bbrace*{\widetilde{e}_j(X_1)} + 2\sum_{i>1}\abs*{\cov\bbrace*{\widetilde{e}_j(X_1), \widetilde{e}_j(X_i)}}\leq 2\sum_{i\geq 1}\abs*{\cov\bbrace*{\widetilde{e}_j(X_1), \widetilde{e}_j(X_i)}}\\
&\leq 2\bbrace*{\sum_{n=0}^\infty 8\alpha^{1/3}(n)}\|\widetilde{e}_j(X_1)\|_{3}\|\widetilde{e}_j(X_1)\|_{3}\leq 64\bbrace*{\sum_{n=0}^\infty \alpha^{1/3}(n)}\|e_j(X_1)\|_{3}\|e_j(X_1)\|_{3}\\
&\leq 64\gamma_1^{1/3}\mu_{3}^2\bbrace*{\sum_{n=0}^\infty \exp\parr*{-\gamma_2n/3}}= \frac{64\gamma_1^{1/3}}{1-\exp\bbrace*{-\gamma_2/3}}\mu_{3}^2.
\end{align*}
Putting together the pieces completes the proof. 
\end{proof}

We now use Lemmas \ref{lemma:approximation}-\ref{lemma:tail} to complete the proof of Theorem \ref{thm:main}.
\begin{proof}[Proof of Theorem \ref{thm:main}]
Fix a $1\leq p\leq m$. Fix some $t > 0$ and $M>0$, and define the event
\begin{align*}
\cal{E} \define \{X_i \in [-M,M]^d \text{ for all } 1\leq i \leq n \}.
\end{align*}
Then, for the prescribed $t$ and $M$, Lemma \ref{lemma:hoeffding_approximation} implies that there exists a symmetric kernel $\td{f} = \td{f}(;t,M)$ such that $\abs*{f-\td{f}(;t,M)}\leq t$ uniformly over all $(x_1^\top,\ldots,x_m^\top)^\top\in[-M,M]^{md}$, and for each $1\leq p\leq m$, $\abs*{f_p - \td{f}_p(;t,M)}\leq Ct$ uniformly over $[-M,M]^{pd}$ for some $C =C(m)$. By definition, this implies that $n^{-p}\abs*{\max_{1\leq k\leq n}\abs*{V_k(f_p)} - \max_{1\leq k\leq n}\abs*{V_k(\td{f}_p)}}\leq Ct$ on the event $\cal{E}$, and thus for any $x > 0$,
\begin{align*}
&\ms\P(n^{-p}\max_{1\leq k\leq n}\abs*{V_k(f_p)}\geq x + Ct)\\
&= \P\parr*{\bbrace*{n^{-p}\max_{1\leq k\leq n}\abs*{V_k(f_p)}\geq x + Ct} \bigcap \cal{E}} + \P\parr*{\bbrace*{n^{-p}\max_{1\leq k\leq n}\abs*{V_k(f_p)}\geq x + Ct} \bigcap \cal{E}^c} \\
&\leq \P\parr*{n^{-p}\max_{1\leq k\leq n}\abs*{V_k(\td{f}_p(;t,M))}\geq x} + n\P(X_1\notin [-M,M]^d).
\end{align*}
Again by Lemma \ref{lemma:hoeffding_approximation}, $\td{f}(;t,M)$ satisfies the conditions of Lemma \ref{lemma:tail} with constants $F = 2^m\nm*{\widehat{f}}_{L_1}$ and $B = 1$. Therefore, applying the trivial bound that $\mu_3 \leq B  = 1$ in Lemma \ref{lemma:tail}, we obtain that
\begin{align*}
\P\Big(n^{-p}\max_{1\leq k\leq n}\abs*{V_k(f_p)}\geq x + Ct\Big) \leq 6\exp\parr*{-\frac{Cnx^{2/p}}{A_p^{1/p} + x^{1/p}M_p^{1/p}}} + n\P(X_1\notin [-M,M]^d),
\end{align*}
where $A_p$ and $M_p$ are defined in \eqref{eq:constant}. Now, note that the first summand on the right hand side does not depend on $M$ or $t$. Accordingly, by first choosing a large enough $M$ that depends only on $x,n,F$, since the measure $P$ considered in this paper is always tight, we obtain that the second term is smaller than an arbitrary small proportion of the first term. Lastly, choosing $t = x$ and adjusting the constant finishes the proof. 
\end{proof}

\section{Proofs of other results}
\label{sec:proof2}

We will only prove Corollaries \ref{cor:smooth_1}-\ref{cor:unif_cont}. The proof of Corollary \ref{cor:unif_cont_shift} is similar to that of Corollary \ref{cor:smooth_1}.

\begin{proof}[Proof of Corollary \ref{cor:smooth_1}]
By inspection of the proof of Theorem \ref{thm:main}, it suffices to prove that when $f_0$ satisfies \eqref{eq:fourier} for some $q\geq 1$, the conclusion of Lemma \ref{lemma:approximation} still holds with $F = 4\nm*{\widehat{f}_0}_{L_1}$ and $B = 1$. Now, following the proof of Lemma \ref{lemma:approximation} with prescribed range $2M$ and approximation error $t$, there exists an $\widetilde{f}_0$ with expansion in the cosine bases $\{\cos\parr*{2\pi u^\top x}\}$ such that $\abs*{f_0-\widetilde{f}_0}\leq t$ uniformly over $[-2M,2M]^d$ and $\widetilde{f}_0$ satisfies the \eqref{eq:expansion_constraint} with constants $F = 2\nm*{\widehat{f}_0}_{L_1}$ and $B = 1$. Let $\widetilde{f}(x, y) \define \widetilde{f}_0(x-y)$. Then, $\abs*{f - \td{f}}\leq t$ uniformly over $[-M,M]^{2d}$, and by the trigonometric identity
\begin{align*}
\cos\bbrace*{2\pi u^\top (x-y)} = \cos\parr*{2\pi u^\top x}\cos\parr*{2\pi u^\top y} + \sin\parr*{2\pi u^\top x}\sin\parr*{2\pi u^\top y},
\end{align*}
$\widetilde{f}$ satisfies \eqref{eq:expansion_constraint} with constants $F = 4\nm*{\widehat{f}_0}_{L_1}$ and $B = 1$. This completes the proof.
\end{proof}

\begin{proof}[Proof of Corollary \ref{cor:smooth_2}]
Again, we only need to reprove Lemma \ref{lemma:approximation}. When $f_0$ is PD, we have by definition that $f_0(0) \geq 0$ and for each $x,y\in\RR^d$, $f_0(x-y) = f_0(y-x)$ implies that $f_0(x) = f_0(-x)$ for any $x\in\RR^d$. This implies that the Fourier transform $\widehat{f}_0$ of $f_0$ is real-valued, and $\widehat{h} = 0$ in the proof of Lemma \ref{lemma:approximation}. Moreover, since $f_0\in L_1(\RR^d)$ as it satisfies \eqref{eq:fourier} for some $q\geq 1$, $f$ equals the inverse Fourier transform of $\widehat{f}$. Thus, by Bochner's theorem (cf. Section 1.4.3, \cite{rudin2017fourier}), $\widehat{f}_0$ is nonnegative and we have $f_0 = I = I_+$ with $m = 1$ in the proof of Lemma \ref{lemma:approximation}. By definition, we have
\begin{align*}
f_0(x) = \int_{\RR^d} \widehat{f}_0(u)e^{2\pi ix^\top u}du = \int_{\RR^d} \abs*{\widehat{f}_0(u)}e^{2\pi ix^\top u}du.
\end{align*}
Letting $x = 0$ in the above equation, we obtain $\nm*{\widehat{f}_0}_{L_1} = f_0(0)$. 

Now consider the case where $\widehat{f}_0$ only has fractional moment. Let $\widetilde{f}_0 \define f_0/f_0(0)$ and denote the Lipschitz constants of $\widetilde{f}_0$ and $f_0$ as $L_{\widetilde{f}_0}$ and $L_{f_0}$, respectively. Then, $L_{\widetilde{f}_0} = L_{f_0}/f_0(0)$. Now we proceed with the proof of Lemma \ref{lemma:approximation} until \eqref{eq:lip_constant}, and replace it with
\begin{align*}
\E\bbrace*{\sup_x \nm*{\nabla_x s_D(x) - \widetilde{f}_0(x)}^q} &\leq \E\bbrace*{\sup_x \parr*{\nm*{\nabla_x s_D(x)}^q + \nm*{\nabla_x \widetilde{f}_0(x)}^q}}\\
&\leq \E\bbrace*{\sup_x\nm*{\nabla_x s_D(x)}^q} + \sup_x\nm*{\nabla_x\widetilde{f}_0(x)}^q\\
&\leq \E\bbrace*{\sup_x\nm*{\nabla_x s_D(x)}^q} + L_{\widetilde{f}_0}^q,
\end{align*}
where $s_D(x) = \sum_{i=1}^D \cos\parr*{2\pi u_i^\top x}/D$ (here we use the notation $s_{D}$ instead of $s_{D_1}$ since in the PD case we only need to approximate the term $I_+$ as argued in the first part of the corollary). Note that the original \eqref{eq:lip_constant} in the proof of Lemma \ref{lemma:approximation} no longer holds as mere fractional moment does not guarantee the exchange of derivative and expectation in its first step. For the first term in the above inequality, we have
\begin{align*}
&\ms\E\bbrace*{\sup_x\|\nabla_xs_{D}(x)\|^q} = \E\fence*{\sup_x\nm*{\frac{1}{D}\sum_{i=1}^{D}2\pi u_i\cos\bbrace*{2\pi\parr*{u_i^\top x}}}^q}\\
&\leq (2\pi)^q\E\parr*{\nm*{\frac{1}{D}\sum_{i=1}^Du_i}^q} \leq (2\pi)^q \E\bbrace*{\parr*{\frac{1}{D}\sum_{i=1}^{D}\nm*{u_i}}^q}\\
&\leq (2\pi)^q \E\bbrace*{D^{-q}\sum_{i=1}^{D}\nm*{u_i}^q} = (2\pi)^qD^{1-q}\E\parr*{\nm*{u_1}^q}.
\end{align*}
Therefore, it holds that
\begin{align*}
\E\bbrace*{\sup_x \nm*{\nabla_x s_D(x) - \widetilde{f}_0(x)}^q} \leq 
 (2\pi)^qD^{1-q}\E\parr*{\|u\|^q} + L_{\widetilde{f}_0}^q.
\end{align*}
Markov inequality now gives
\begin{align*}
\P\bbrace*{\sup_x\nm*{\nabla_x\parr*{s_D(x) - \widetilde{f}_0(x)}}\geq \frac{t}{2r}} \leq \parr*{\frac{2r}{t}}^q\bbrace*{(2\pi)^qD^{1-q}\E\parr*{\|u\|^q} + L_{\widetilde{f}_0}^q}. 
\end{align*}
Proceeding with the proof of Lemma \ref{lemma:approximation}, we obtain 
\begin{align*}
\P\bbrace*{\sup_{x\in\cal{M}}\abs*{s_D(x) - \widetilde{f}_0}\geq t} \leq \parr*{\frac{2r}{t}}^q\bbrace*{(2\pi)^qD^{1-q}\E\parr*{\|u\|^q} + L_{\widetilde{f}_0}^q} + \parr*{\frac{c\,\diam(\cal{M})}{r}}^{md}\exp\parr*{-\frac{Dt^2}{8}}.
\end{align*}
Writing the right-hand side of the above inequality in the form $\kappa_1r^{-md} + \kappa_2r^q$ and letting $r = (\kappa_1/\kappa_2)^{1/(q+md)}$, we obtain
\begin{align*}
\P\bbrace*{\sup_{x\in\cal{M}}\abs*{s_D(x) - \widetilde{f}_0}\geq t} \leq 2\parr*{\frac{2c\,\diam(\cal{M})}{\varepsilon}}^{\frac{qmd}{q+md}}\bbrace*{(2\pi)^qD^{1-q}\E\parr*{\|u_1\|^q} + L_{\widetilde{f}_0}^q}^{\frac{md}{q+md}}\exp\parr*{-\frac{D\varepsilon^2}{8}\frac{q}{q+md}}.
\end{align*}
For any $t > 0$, we can choose large enough $D = D(t)$ such that the right-hand side of the above inequality is arbitrarily small. The proof is complete.
\end{proof}

\begin{proof}[Proof of Corollary \ref{cor:unif_cont}]
Let $K(\cdot):\RR^{md}\rightarrow\RR$  be the standard $md$-variate Gaussian density defined as $K(x) \define \exp(-\|x\|^2/2)(2\pi)^{-md/2}$, and $K_h(x) = K(x/h)h^{-md}$ for some positive constant $h$. Define $f_h(x) \define (f*K_h)(x)$. Then, it holds that
\begin{align*}
\abs*{f_h(x) - f(x)} &= \abs*{\int_{\RR^{md}} (2\pi)^{-md/2}\exp\parr*{-\frac{\|y\|^2}{2}}\bbrace*{f(x-yh) - f(x)}dy}\\
&\leq \int_{\RR^{md}} (2\pi)^{-md/2}\exp\parr*{-\frac{\|y\|^2}{2}}\abs*{f(x-yh) - f(x)}dy.
\end{align*}
Denote the upper bound of $f$ as $M_f$. Then, for any $t > 0$, there exists some positive constant $A = A(M_{f},m,d,t)$ such that
\begin{align*}
&\ms\int_{\parr*{[-A,A]^{md}}^c}(2\pi)^{-md/2}\exp\parr*{-\frac{\|y\|^2}{2}}\abs*{f(x-yh) - f(x)}dy\\
&\leq 2M_{f}\int_{\parr*{[-A,A]^{md}}^c}(2\pi)^{-md/2}\exp\parr*{-\frac{\|y\|^2}{2}}dy \leq t/4.
\end{align*}
Inside $[-A,A]^{md}$, using the uniform continuity of $f$, there exists some $h = h(M_{f},m,d,t)$, such that
\begin{align*}
\int_{[-A,A]^{md}}(2\pi)^{-md/2}\exp\parr*{-\frac{\|y\|^2}{2}}\abs*{f(x-yh) - f(x)}dy \leq t/4.
\end{align*}
Putting together the pieces, it holds that for any $t > 0$, there exists some $h = h(M_{f},m,d,t)$ such that $\|f_h-f\|_\infty\leq t/2$.
 Since both $f$ and $K_h$ belong to $L_1(\RR^{md})$, their Fourier transforms exist. It can be readily checked that $\widehat{K}_h(u) = \exp\parr*{-2\pi^2h^2\|u\|^2}$, and thus
\begin{align*}
\widehat{f}_h(u) = \widehat{f}(u)\cdot \widehat{K}_h(u) = \widehat{f}(u)\exp\parr*{-2\pi^2h^2\|u\|^2}.
\end{align*}
Using the relation $\|f*g\|_{L_q} \leq \|f\|_{L_q}\|g\|_{L_1}$ for any $q\geq 1$ and $f\in L_q(\RR^{md}), g\in L_1(\RR^{md})$ and the fact that $K_h\in L_1(\RR^{md})$, it holds that $f_h\in L_1(\RR^{md})$. Moreover, it can readily checked that 
\begin{align*}
\mu_q^q\parr*{\widehat{f}_h} = \int_{\RR^{md}}\abs*{\widehat{f}_h(u)}\|u\|^qdu = \int_{\RR^{md}}\abs*{\widehat{f}(u)}\|u\|^q\exp\parr*{-2\pi^2h^2\|u\|^2}du < \infty
\end{align*}
for any $q\geq 1$. Therefore, by Lemma \ref{lemma:approximation}, for any given $M > 0$ and $t > 0$, we can find an approximating kernel $\widetilde{f}_h = \widetilde{f}_h(t)$ such that $\abs*{\widetilde{f}_h-f_h}\leq t/2$ uniformly over $[-M,M]^{md}$, and $\widetilde{f}_h$ further satisfies \eqref{eq:expansion_constraint} with constants $F = 2^m\nm*{\widehat{f}_h}_{L_1}$ and $B = 1$.
Choosing $h = h(M_{f},m,d,t/2)$, by the triangle inequality, we have
\begin{align*}
\abs*{f - \widetilde{f}_h} \leq t/2 + t/2 = t
\end{align*}
uniformly over $[-M,M]^{md}$. Lastly, we upper bound the term $\nm*{\widehat{f}_h}_{L_1}$. To this end, we have 
\begin{align*}
\nm*{\widehat{f}_h}_{L_1} &= \int_{\RR^{md}} \abs*{\widehat{f}(u)}\exp(-2\pi^2h^2\|u\|^2)du \leq L\int_{\RR^{md}} \frac{1}{1+\|u\|^{md+\varepsilon}}\exp(-2\pi^2h^2\|u\|^2)du.
\end{align*}
Using polar coordinates, it holds that
\begin{align*}
\nm*{\widehat{f}_h}_{L_1} &\leq \Gamma(md)L\int_0^\infty \frac{r^{md-1}}{1+r^{md+\varepsilon}}\exp(-2\pi^2h^2r^2)dr\\
&\leq \Gamma(md)L\parr*{1 + \int_1^\infty \frac{1}{r^{1+\varepsilon}}dr} \\
&= (1+\varepsilon^{-1})\Gamma(md)L.
\end{align*}
This completes the proof.
\end{proof}

\bibliography{reference}

\end{document}